\documentclass[review, 3p]{elsarticle}
\usepackage{lineno, hyperref}
\usepackage{amsfonts} 
\usepackage{graphicx} 
\usepackage{amsmath} 
\usepackage{color}
\usepackage{ntheorem}
\usepackage{tikz}
\newtheorem{thm11}{Theorem}[section]
\newtheorem{theorem}[thm11]{Theorem}
\newtheorem{lemma}[thm11]{Lemma}
\newtheorem{proposition}[thm11]{Proposition}

\newtheorem{conjecture}[thm11]{Conjecture}
\newtheorem{definition}[thm11]{Definition}

\newenvironment{proof}{\emph{Proof. }}{\hspace*{\fill}$\Box$\par\vskip2ex}
\numberwithin{equation}{section}
\usepackage{graphicx}
\graphicspath{{Figures/}}
\usetikzlibrary{calc}

\journal{Discrete Mathematics}

\usepackage{tikz}
\usetikzlibrary{decorations.markings}
\tikzstyle{vertex} = [circle, draw, inner sep = 0pt, minimum size = 6pt]
\usepackage{caption}
\usepackage{subcaption}
\usepackage{amssymb}

\newcommand{\RS}{\operatorname{RS}}

\begin{document}

\begin{frontmatter}

\title{Characterization of Complete Bipartite Graphs via Resistance Spectra }


\author{Xiang-Yang Liu}
\fntext[myFunding]{Funding: Partially supported by the Excellent University Research and Innovation Team in Anhui Province (No. 2024AH010002) and the
	University Natural Science Research Project of Anhui Province, China (No. 2024AH050069)
}

\author{Xiang-Feng Pan\fnref{myFunding}\corref{mycorrespondingauthor}}
\cortext[mycorrespondingauthor]{Corresponding author}
\ead{xfpan@ahu.edu.cn}

\author{Yong-Yi Jin}

\author{Li-Cheng Li}

\address{School of Mathematical Sciences, Anhui University, Hefei, Anhui, 230601, P. R. China}
\begin{abstract}
The notion of resistance distance, introduced by Klein and Randi\'c, has become a fundamental concept in spectral graph theory and network analysis, as it captures both the structural and electrical properties of a graph. The associated resistance spectrum serves as a graph invariant and plays an important role in problems related to graph isomorphism. For an undirected graph $G=(V,E)$, the resistance distance $R_G(u,v)$ between two distinct vertices $u$ and $v$ is defined as the effective resistance between them when each edge of $G$ is replaced by a $1\,\Omega$ resistor. The multiset of all resistance distances over unordered pairs of distinct vertices is called the \emph{resistance spectrum} of $G$, denoted by $\operatorname{RS}(G)$. A graph $G$ is said to be \emph{determined by its resistance spectrum} if, for any graph $H$, the equality $\operatorname{RS}(H)=\operatorname{RS}(G)$ implies that $H$ is isomorphic to $G$. Complete bipartite graphs, denoted by $K_{m,n}$, are highly symmetric and constitute an important class of graphs in graph theory. In this paper, by exploiting properties of resistance distances, we prove that the complete bipartite graphs $K_{n,n}$, $K_{n,n+1}$, $K_{2,n}$, and $K_{m,n}$ with $m>3n+1$ are uniquely determined by their resistance spectra.
\end{abstract}


\end{frontmatter}

\section{Introduction}\label{sec:introduction}
The notion of resistance distance was first introduced by Klein and Randi\'c in 1993 in the context of electrical network theory \cite{KR1993}. For a given graph \(G\), the resistance distance \(R_{G}(x, y)\) between two vertices \(x\) and \(y\) corresponds to the effective resistance between them when \(G\) is regarded as an electrical network in which every edge represents an \(1\,\Omega\) resistor. If there is no ambiguity, we denote $R_G(x,y)$ by $R(x,y)$ for notational convenience. The resistance diameter \(D_{r}(G)\) of \(G\) is then defined as the maximum value of \(R_{G}(x, y)\) taken over all vertex pairs in \(G\) \cite{SARDAR2020123076}.  As an important graph invariant, the resistance distance offers a more refined characterization of vertex connectivity compared to conventional graph distance measures. This parameter has attracted considerable attention across multiple disciplines, including mathematics, chemistry, and ecology. The resistance distance captures connectivity patterns that resemble wave or fluid propagation between vertices, providing advantages over ordinary distance metrics \cite{klein2002resistance}. 
The measure also supports the analysis of random algorithms \cite{chandra1989electrical} and has found uses in various practical areas, including motion tracking \cite{qiu2006robust}, social network analysis \cite{liben2003link}, and image segmentation \cite{qiu2005image}.  
Recent advances in the subject are thoroughly reviewed in \cite{zhu2024resistanceDirected, sardar2024clawfree} and related works.

The resistance spectrum $\operatorname{RS}(G)$ of a graph $G$ is defined as the multiset of resistance distances between all pairs of distinct vertices in $G$. A graph $G$ is determined by its resistance spectrum if, for any graph $H$, the identity $\operatorname{RS}(H)=\operatorname{RS}(G)$ implies that $H$ is isomorphic to $G$. The concept of using the resistance spectrum to study the isomorphism of the graph is attributed to Baxter, who conjectured that two graphs are isomorphic precisely when their resistance spectra coincide \cite{Baxter1999RS22}. However, this conjecture was later refuted by explicit counterexamples showing that non-isomorphic graphs can share identical resistance spectra \cite{Rickard1999Counterexample23, WeissteinRSEquivalent}. Despite this, the resistance spectrum remains a powerful invariant that captures global structural information, motivating continued research into which graph families are uniquely characterized by it. Recent work has identified many such families. Complete graphs and cycles were shown to be determined by their resistance spectra in \cite{Xu2021Resistance}. This result was extended to paths and lollipop graphs in \cite{Xu2022Resistance}. Further progress includes star graphs, double star graphs, and generalized octopus graphs \cite{zhou2025determination}, as well as double starlike trees, sandglass graphs, kite graphs, and chained oxide networks \cite{xing2025determination}. Studies also confirm the determinability of comb graphs, $T$-shape trees, $\infty$-graphs, and dumbbell graphs via their resistance spectra \cite{10.1088/1402-4896/adf7df}. A method for constructing non-isomorphic graphs with the same resistance spectrum is presented in \cite{XU2025114284}, offering insight into the limitations of this invariant. 
 	
 	A graph \(G=(V,E)\) is said to be a complete bipartite graph if its vertex set \(V\) can be partitioned into two nonempty disjoint subsets \(V_1\) and \(V_2\) (i.e., \(V = V_1 \cup V_2\) and \(V_1 \cap V_2 = \emptyset\)) such that every edge of \(G\) joins a vertex in \(V_1\) with a vertex in \(V_2\) (so \(G\) is bipartite), and moreover every vertex in \(V_1\) is adjacent to every vertex in \(V_2\). If \(|V_1| = m\) and \(|V_2| = n\), the graph is denoted by \(K_{m,n}\).  As a core subject in graph theory, complete bipartite graphs are widely studied across discrete mathematics and computer science because of their clear structure and natural suitability for modeling bipartite systems. Researchers have advanced their understanding from multiple perspectives: Othman and Berger \cite{OTHMAN2024113865} addressed the almost fair partitioning of perfect matchings in matching theory, realizing refined evaluation of matching quality; Zhao et al. \cite{inproceedings} explored graph orientation and proposed an algorithm for constructing orientations with specific extremal properties; Bekos et al. \cite{Bekos2023OnTD} introduced deque-based models to linear layout problems, establishing new bounds for the complexity of layout of complete bipartite graphs. Together, these works that span structural properties, algorithm design, and complexity analysis underscore the profound theoretical potential and the lasting research value of complete bipartite graphs.
 	
 In this paper, we studied the properties of the resistance distance in bipartite graphs. Using existing resistance distance properties, we proved that $K_{n,n}$, $K_{n,n+1}$, and $K_{m,n}$ with $m > 3n + 1$ can be determined by their resistance spectra. Furthermore, we conjecture that all complete bipartite graphs can be determined by their resistance spectra.
 	
\section{Preliminary knowledge}
 This section establishes the theoretical framework for our analysis by presenting core concepts from electrical network theory along with supporting lemmas. These fundamental elements are integral to the validity of our subsequent proofs. The scope of this work is confined to undirected simple graphs. For undefined terminology and notation, the reader is directed to Bondy and Murty's \textit{Graph Theory} \cite{BM2008graph}.

	Let \( G = (V(G), E(G)) \) represent a connected simple undirected graph, where \( V(G) \) and \( E(G) \) correspond to the vertex set and edge set respectively. For any vertex \( v \in V(G) \), \( d_G(v) \) signifies its degree, while \( N_G(v) \) denotes its neighborhood. When no confusion arises, we also denote $d_G(u)$ and $N_G(u)$ simply by $d(u)$ and $N(u)$, respectively. 
 
 We now introduce the lemmas and principles required to prove our main results.
\begin{lemma}[Series Principle]\label{lemma:Serirs Principle}
	For a series circuit comprising $n$ resistors with individual resistances $r_1, r_2, \dots, r_n$, the entire network can be replaced by an equivalent single resistor. The equivalent resistance $R$ is equal to the summation of all individual resistances:
	\[R = \sum_{i=1}^{n} r_i = r_1 + r_2 + \cdots + r_n.\]
\end{lemma}
\begin{lemma}[Parallel Principle]\label{lemma:Parallel principle}
	When $n$ resistors with resistances $r_1, r_2, \dots, r_n$ are connected in parallel, 
	the entire configuration can be replaced by an equivalent single resistor. 
	The equivalent resistance $R$ satisfies the following:
	\[
	\frac{1}{R} = \sum_{k=1}^{n} \frac{1}{r_k} = \frac{1}{r_1} + \frac{1}{r_2} + \cdots + \frac{1}{r_n}.
	\]
\end{lemma}
\begin{lemma}[The elimination principle]\cite{klein2002resistance}\label{lemma:The elimination principle}
 Consider a connected graph $G$ and a block $B$ of $G$ containing exactly one cut vertex $w$. 
 Let $G' = G - (V(B) \setminus \{w\})$ be the graph obtained by removing all the vertices of $B$ except $w$. 
 Then for any pair of vertices $u, v \in V(G')$, the effective resistance satisfies the following:
 \[
 R_G(u, v) = R_{G'}(u, v).
 \]
 \end{lemma}
\begin{lemma}[Triangle inequality]\cite{KR1993}\label{lemma:Triangle inequality}
	For any connected graph $G$ and every triple of vertices $u, v, w \in V(G)$, 
	\[ R_G(u, v) + R_G(v, w) \geq R_G(u, w). \]
\end{lemma}
\begin{definition}
Consider two electrical networks $G$ and $H$ that share a common subset of vertex $S \subseteq V(G) \cap V(H)$. 
These networks are precisely termed $S$-equivalent when 
\[ R_G(u, v) = R_H(u, v) \quad \text{for all} \quad u, v \in S. \]
\end{definition}
\begin{proposition}[Principle of substitution]\label{proposition:substitution}
	Suppose $G$ is an electrical network with vertex set $V(G)$, and $H$ is a connected subgraph of $G$. 
	Let $H^{*}$ be a graph that satisfies $V(H) \subseteq V(H^{*})$ and such that $H$ and $H^{*}$ are $V(H)$-equivalent. 
	Form $G^{*}$ by replacing $H$ with $H^{*}$ in $G$. 
	Then for every pair of vertices $u, v \in V(G)$,
	\[
	R_G(u, v) = R_{G^{*}}(u, v),
	\]
	this implies that $G$ and $G^{*}$ are $V(G)$-equivalent.
\end{proposition}

\begin{lemma} \cite{KR1993}\label{lemma:RDcut}
Suppose $w$ is a cut vertex in a connected graph $G$. 
If $u$ and $v$ lie in distinct connected components after removing $w$ from $G$, then
\[
R_G(u, v) = R_G(u, w) + R_G(w, v).
\]
\end{lemma}
\begin{lemma}\cite{H2010}\label{lamma:n-2}
	For a connected graph $G$ and any distinct vertices $u, v \in V(G)$, 
	\[
	d_G(u) R_G(u, v) + \sum_{z \in N_G(u)} \big( R_G(z, u) - R_G(z, v) \big) = 2.
	\]
\end{lemma}
\begin{lemma}\cite{foster1949average}\label{lemma:n-1}
In any connected graph $G$ with $n$ vertices,
\[
\sum_{uv \in E(G)} R_G(u, v) = n - 1.
\]
\end{lemma}
\begin{lemma}\label{lemma:xiajie}
Let $ G $ be a simple connected graph. Then the resistance distance between any two vertices$  u $ and $ v$  satisfies \[
R_{G}(u, v) \geq \frac{1}{d_{G}(u)+1}+\frac{1}{d_{G}(v)+1}
\]
 and the equality holds if and only if $ u $ and $ v $ are adjacent and have the same set of neighbors in $ G-\{u, v\} $.
\end{lemma}
\begin{lemma}[Rayleigh's monotonicity law]\cite{J1871}\label{lemma:rayleigh}
    Consider two electrical networks $N$ and $N'$ where $N'$ is obtained by maintaining or increasing the resistance of each edge relative to $N$. 
Then for all vertices $u, v \in V(N')$,
\[
R_N(u, v) \leq R_{N'}(u, v).
\]
\end{lemma}

By Rayleigh's monotonicity law, we can get the following conclusions.
\begin{proposition} \label{proposition:subgraph}
For any subgraph $H$ of a graph $G$ and for every pair of vertices $u, v \in V(H)$,
\[
R_G(u, v) \leq R_H(u, v).
\]
\end{proposition}
\begin{proposition}\label{proposition:cycle}\cite{xing2025determination}
Suppose $G$ is a simple connected graph containing at least one cycle. 
If $u$ and $v$ are adjacent vertices on some cycle of $G$, then
\[
R_G(u, v) < 1.
\]
\end{proposition}

\section{Main results}\label{sec:Main}

In this section, we will prove in sequence that $K_{n,n}$, $K_{n,n+1}$, and $K_{m,n}$ (when $m > 3n + 1$) can be determined by their resistance spectra. Before proceeding, it is necessary to determine the resistance spectrum of the complete bipartite graph $K_{m,n}$. Suppose that the bipartition of $K_{m,n}$ is $X$ and $Y$, where $X = \{x_1, x_2, \dots, x_m\}$ and $Y = \{y_1, y_2, \dots, y_n\}$. According to \cite{klein2002resistance}, for any two vertices $u$ and $v$ in $K_{m,n}$, the resistance distance is given by the following formulas:

\[
	R(u, v) = \begin{cases}
		\dfrac{2}{n}, & \text{if } u, v \in X, \\[10pt]
		\dfrac{2}{n}, & \text{if } u, v \in X, \\[10pt]
		\dfrac{1}{m}+\dfrac{1}{n}-\dfrac{1}{mn}, & \text{if } u\in X, v \in Y.
	\end{cases}
\]
Consequently, the resistance spectrum of $K_{m,n}$ can be fully derived.
\[
\RS(K_{m,n})=\left\lbrace \left[\frac{2}{m} \right]^{\frac{n(n-1)}{2}}, \left[ \frac{1}{m}+\frac{1}{n}-\frac{1}{mn}\right]^{mn}, \left[\frac{2}{n} \right]^{\frac{m(m-1)}{2}} \right\rbrace.
\]

 Since all elements in the multiset $\RS(K_{m,n})$ are finite, $H$ must be a connected graph when it has the same resistance spectrum as $K_{m,n}$. We will now prove that $K_{n,n}$ can be determined by its resistance spectrum.
\begin{theorem}\label{knn}
	The complete bipartite graph $K_{n,n}$ is determined by its resistance spectrum.
\end{theorem}
\begin{proof}
	The resistance spectrum of the complete bipartite graphs $K_{n,n}$ is given by
	\[
	\RS(K_{n,n})=\left\lbrace \left[\frac{2}{n} \right]^{n(n-1)}, \left[\frac{2}{n}-\frac{1}{n^2} \right]^{n^2} \right\rbrace.
	\]
	
	If $H$ has the same resistance spectrum as the complete bipartite graph $K_{n,n}$, we prove below that $H$ is isomorphic to $K_{n,n}$.
	
	First, we claim that for any $u,v\in V(H)$, if $R_H(u,v)=\frac{2}{n}$, then we must have $uv\notin E(H)$.
	
	By contradiction, if $uv\in E(H)$, then applying Lemma \ref{lamma:n-2} to vertices $u$ and $v$, we have the following.
	\[
	\frac{2}{n}d(u)+R(u,v)+\sum_{z \in N(u)\backslash \{v\}}\left( R(u,z)-R(v,z) \right)=2.
	\]
	It is easy to see that there must exist an integer $s_1$ such that $\frac{s_1}{n^2}=\sum_{z \in N(u)\backslash \{v\}}\left( R(u,z)-R(v,z) \right)$, where $-2n+2\leq -d(u)+1 \leq s_1 \leq d(u)-1 \leq 2n-2$. That is,
	\[
	2n(d(u)-n+1)+s_1=0.
	\]
	Thus $2n\mid s_1$, so we have $s_1=0$ and $d(u)=n-1$. Similarly, $d(v)=n-1$.
	
	Hence, we have 
	\[
	R(u,v)=\frac{2}{n}=\frac{1}{d(u)+1}+\frac{1}{d(v)+1}.
	\]
	According to the equality condition in Lemma \ref{lemma:xiajie}, we have $N(u)\backslash \{v\} = N(v)\backslash \{u\}=\{u_1,u_2,\ldots,u_{n-2}\}$. Next we discuss two cases: (\romannumeral1) for every $1\leq i\leq n-2$, we have $R(u,u_i)=\frac{2}{n}$ and (\romannumeral2) there exists $1\leq i\leq n-2$ such that $R(u,u_i)=\frac{2}{n}-\frac{1}{n^2}$.
	
	(\romannumeral1) If for every $1\leq i\leq n-2$, we have $R(u,u_i)=\frac{2}{n}$, then similarly to the discussion for $u,v$, applying Lemma \ref{lamma:n-2} and Lemma \ref{lemma:xiajie} to $u,u_i$, we have
	\[
	N(u)\backslash \{u_i\} = N(u_i)\backslash \{u\}
	\]
	for every $1\leq i\leq n-2$. That is, $H$ contains a complete subgraph composed of $\{u_1,u_2,\ldots,u_{n-2},u_{n-1}=u,u_n=v\}=W$. Since $|W|=n$, the vertex set $V(H)\setminus W$ is not empty. Take any $x\in V(H)\setminus W$; clearly $x$ is not adjacent to $u$, and due to $N(u)\backslash \{u_i\} = N(u_i)\backslash \{u\}$, $x$ is not adjacent to any vertex in $W$. By the arbitrariness of $x$, it follows that $H$ is a disconnected graph, 
	which leads to a contradiction.

	(\romannumeral2) If there exists $1\leq i\leq n-2$ such that $R(u,u_i)=\frac{2}{n}-\frac{1}{n^2}$, then applying Lemma \ref{lamma:n-2} to $u,u_i$, we have
	\[
	\left( \frac{2}{n}-\frac{1}{n^2} \right)d(u)+\frac{2}{n}-\frac{1}{n^2}+\sum_{z \in N_H(u)\backslash \{u_i\}}\left( R_H(u,z)-R_H(u_i,z) \right)=2.
	\]
	That is,
	\[
	\left( \frac{2}{n}-\frac{1}{n^2} \right)(n-1)+\frac{2}{n}-\frac{1}{n^2}+\frac{n-2}{n^2}\geq2.
	\]
	Hence, we have $-2\geq 0$, a contradiction.
	
	From the above, we have shown that for any $u,v\in V(H)$, if $R_H(u,v)=\frac{2}{n}$, then we must have $uv\notin E(H)$. Equivalently, if $uv\in E(H)$, then we must have $R(u,v)=\frac{2}{n}-\frac{1}{n^2}$. Using Lemma \ref{lemma:n-1}, we have the following.
	\[
	\sum_{uv \in E(H)} R_(u, v) = 2n - 1.
	\]
	Thus $(\frac{2}{n}-\frac{1}{n^2})|E(H)|=2n-1$. Therefore, the graph $H$ has $n^2$ edges. Moreover, since the multiplicity of $\frac{2}{n}-\frac{1}{n^2}$ in $RS(H)$ is also $n^2$, for every $uv\notin E(H)$, we have $R(u,v)=\frac{2}{n}$. Taking any two non-adjacent vertices $u,v$ in $V(H)$, using Lemma \ref{lamma:n-2}, we have
	\[
	\frac{2}{n}d(u)+\sum_{z \in N(u)}\left( R(u,z)-R(v,z) \right)=2.
	\]
	It is not hard to see that there exists an integer $s_2$ such that $\frac{s_2}{n^2}=\sum_{z \in N(u)}\left( R(u,z)-R(v,z) \right)$, where $-2n+1\leq -d(u) \leq s_2 \leq d(u) \leq 2n-1$. Then we have
	\[
	\frac{2}{n}d(u)+\frac{s_2}{n^2}=2.
	\]
	Hence, we have $2n(d(u)-n)+s_2=0$. So $s_2=0, d(u)=n$, similarly we have $d(v)=n$.
	
	Finally, take any $u,v$ satisfying $R(u,v)=\frac{2}{n}$, then $d(u)=d(v)=n$. Denote $N(u)=\{v_1,v_2,\ldots,v_n \}$, and $U=V(H)\backslash N(u)=\{u_1=u,u_2,\ldots,u_n \}$ is the set consisting of $u$ and the vertices not adjacent to $u$. Add an edge $uu_2$ to $H$ to obtain a new graph $H'$, then by the parallel principle and the principle of substitution we have $R_{H'}(u,u_2)=\frac{2}{n+2}$. Thus,
	\[
	R_{H'}(u,u_2)=\frac{1}{d_{H'}(u)+1}+\frac{1}{d_{H'}(u_2)+1}.
	\]
	By Lemma \ref{lemma:xiajie}, we have $N_{H'}(u)\backslash \{u_2\}=N_{H'}(u_2)\backslash \{u\}$, so we have $N_{H}(u)=N_{H}(u_2)$. Similarly, $N_{H}(u)=N_{H}(u_i), 3\leq i \leq n$. Thus, we have found a complete bipartite subgraph $H''$ of $H$ with bipartition $U$ and $N(u)$, and because $|E(H)|=|E(H'')|$, we have $H=H''\cong K_{n,n}$.
\end{proof}

Similarly to the proof of Theorem \ref{knn}, we have the following result.

\begin{theorem}\label{knn+1}
	The complete bipartite graphs $K_{n,n+1}$ is determined by its resistance spectrum.
\end{theorem}
\begin{proof}
The resistance spectrum of the complete bipartite graph $K_{n,n+1}$ is given by
\[
\RS(K_{n,n+1})=\left\lbrace \left[\frac{2}{n+1} \right]^{\frac{n(3n+1)}{2}}, \left[\frac{2}{n} \right]^{\frac{n(n+1)}{2}} \right\rbrace.
\]

If $H$ has the same resistance spectrum as the complete bipartite graph $K_{n,n+1}$, we prove below that $H$ is isomorphic to $K_{n,n+1}$.

First, we claim that for any $u,v\in V(H)$, if $R_H(u,v)=\frac{2}{n}$, then we must have $uv\notin E(H)$.

By contradiction, if $uv\in E(H)$, then applying Lemma \ref{lamma:n-2} to vertices $u$ and $v$, we have
\[
\frac{2}{n}d(u)+R(u,v)+\sum_{z \in N(u)\backslash \{v\}}\left( R(u,z)-R(v,z) \right)=2.
\]
It is easy to see that there must exist an integer $s_1$ such that $\frac{2s_1}{n(n+1)}=\sum_{z \in N(u)\backslash \{v\}}\left( R(u,z)-R(v,z) \right)$, where $-2n+1\leq -d(u)+1 \leq s_1 \leq d(u)-1 \leq 2n-1$. That is,
\[
(n+1)(d(u)-n+1)+s_1=0.
\]
Thus $n+1\mid s_1$, so we have $s_1=0$ and $d(u)=n-1$. Similarly, $d(v)=n-1$.

Hence, we have 
\[
R(u,v)=\frac{2}{n}=\frac{1}{d(u)+1}+\frac{1}{d(v)+1}.
\]
According to the equality condition in Lemma \ref{lemma:xiajie}, we have $N(u)\backslash \{v\} = N(v)\backslash \{u\}=\{u_1,u_2,\ldots,u_{n-2}\}$. Next we discuss two cases: (\romannumeral1) for every $1\leq i\leq n-2$, we have $R(u,u_i)=\frac{2}{n}$ and (\romannumeral2) there exists $1\leq i\leq n-2$ such that $R(u,u_i)=\frac{2}{n}-\frac{1}{n^2}$.

Following the proof of Theorem \ref{knn}, we can show that case (\romannumeral1) does not exist.

If there exists $1\leq i\leq n-2$ such that $R(u,u_i)=\frac{2}{n+1}$, then applying Lemma \ref{lamma:n-2} to $u,u_i$, we have
\[
\frac{2}{n+1}d(u)+\frac{2}{n+1}+\sum_{z \in N_H(u)\backslash \{u_i\}}\left( R_H(u,z)-R_H(u_i,z) \right)=2.
\]
That is,
\[
\frac{2}{n+1}(n-1)+\frac{2}{n+1}+\frac{2(n-2)}{n(n+1)}\geq2.
\]
Hence, we have $-4\geq 0$, a contradiction.

From the above, we have shown that for any $u,v\in V(H)$, if $R_H(u,v)=\frac{2}{n}$, then we must have $uv\notin E(H)$. Equivalently, if $uv\in E(H)$, then we must have $R(u,v)=\frac{2}{n+1}$. Using Lemma \ref{lemma:n-1}, we have
\[
\sum_{uv \in E(H)} R_(u, v) = 2n.
\]
Thus $\frac{2}{n+1}|E(H)|=2n$. Therefore, the graph $H$ has $n(n+1)$ edges. Choosing two arbitrary vertex $u$ and $v$ satisfying $R(u,v)=\frac{2}{n}$, using Lemma \ref{lamma:n-2}, we have
\[
\frac{2}{n}d(u)+\sum_{z \in N(u)}\left( R(u,z)-R(v,z) \right)=2.
\]
It is not hard to see that there exists an integer $s_2$ such that $\frac{2s_2}{n(n+1)}=\sum_{z \in N(u)}\left( R(u,z)-R(v,z) \right)$, where $-2n+1\leq -d(u) \leq s_2 \leq 0$. Then we have
\[
\frac{2}{n}d(u)+\frac{2s_2}{n(n+1)}=2.
\]
Hence, we have $(n+1)(d(u)-n)+s_2=0$. This equation admits two possible solutions: $s_2 = 0$, $d(u) = n$ and $s_2 = -(n+1)$, $d(u) = n+1$. Similarly, the possible values for $d(v)$ are $n$ and $n+1$. We now prove that the case $s_2 = -(n+1)$, $d(u) = n+1$ is impossible:

Assume, by contradiction, that $s_2 = -(n+1)$ and $d(u) = n+1$. Then for every $x \in N(u)$, we have $R(u,x) = \frac{2}{n+1}$ and $R(v,x) = \frac{2}{n}$. This implies that for every $x \in N(u)$, $xv \notin E(H)$, and hence $d(v) \leq n-2$. However, this contradicts the fact that $d(v)$ can only be $n$ or $n+1$. Therefore, we must have $d(u) = n$. Similarly, we obtain $d(v) = n$.

Denote $N(u)=\{v_1,v_2,\ldots,v_n \}$, and $U=V(H)\backslash N(u)=\{u_1=u,u_2,\ldots,u_{n+1} \}$ is the set consisting of $u$ and the vertices not adjacent to $u$. Suppose that there exists some $i$ with $2 \leq i \leq n+1$ such that $R(u,u_i) = \frac{2}{n+1}$. Applying Lemma \ref{lamma:n-2} to $u$ and $u_i$ yields:
\[
2>\frac{2n}{n+1}+0\geq R(u,u_i)d(u)+\sum_{z \in N(u)}\left( R(u,z)-R(u_i,z) \right)=2.
\]
This leads to a contradiction. Therefore, $R(u,u_i) = \frac{2}{n}$ for all $2 \leq i \leq n+1$.

Add an edge $uu_2$ to $H$ to obtain a new graph $H'$, then by the parallel principle and the principle of substitution we have $R_{H'}(u,u_2)=\frac{2}{n+2}$. Thus
\[
R_{H'}(u,u_2)=\frac{1}{d_{H'}(u)+1}+\frac{1}{d_{H'}(u_2)+1}.
\]
By Lemma \ref{lemma:xiajie}, we have $N_{H'}(u)\backslash \{u_2\}=N_{H'}(u_2)\backslash \{u\}$, so we have $N_{H}(u)=N_{H}(u_2)$. Similarly, $N_{H}(u)=N_{H}(u_i), 3\leq i \leq n+1$. Thus, we have found a complete bipartite subgraph $H''$ of $H$ with bipartition $U$ and $N(u)$, and because $|E(H)|=|E(H'')|$, we have $H=H''\cong K_{n,n+1}$.
\end{proof}

In \cite{zhou2025determination}, it has been proven that star graphs can be determined by their resistance spectra. Since a star graph is essentially a complete bipartite graph $K_{1,n}$, correspondingly, we have proven that $K_{2,n}$ can also be determined by its resistance spectrum.

\begin{theorem}\label{k2n}
	The complete bipartite graph $K_{2,n}$ is determined by its resistance spectrum.
\end{theorem}
\begin{proof}
	The resistance spectrum of the complete bipartite graph $K_{2,n}$ is given by
	\[
	\RS(K_{2,n})=\left\lbrace \left[\frac{2}{n} \right]^{1}, \left[\frac{1}{2}+\frac{1}{2n} \right]^{2n}, \left[1\right]^{\frac{n(n-1)}{2}} \right\rbrace.
	\]
	
	If $H$ has the same resistance spectrum as the complete bipartite graph $K_{2,n}$, we prove below that $H$ is isomorphic to $K_{2,n}$. By Theorems \ref{knn} and \ref{knn+1}, the conclusion holds for $n \leq 3$. In the following, we only consider the case $n \geq 4$.
	
	First, we prove that for any $u, v \in V(H)$, if $R(u, v) = 1$, then $uv \notin E(H)$.
	
	Suppose, for contradiction, that $uv \in E(H)$ and $R(u, v) = 1$. Then there exists a block $B$ of $H$ such that $u, v \in V(B)$. If $\{u, v\} = V(B)$, i.e., $B$ consists solely of the edge $uv$, then since $H$ is connected, there exists a vertex $x$ such that $x \in N(u) \setminus \{v\}$ or $x \in N(v) \setminus \{u\}$. Without loss of generality, assume $x \in N(u) \setminus \{v\}$. Then by Lemma \ref{lemma:RDcut}, we have
	\[
	R(v,x)=R(v,u)+R(u,x)>1,
	\]
	which contradicts the fact that the resistance distance between any two vertices in $H$ is at most $1$. If $\{u, v\} \subsetneqq V(B)$ and $B$ is $2$-connected, then there exists a cycle $C$ of length t containing the edge $uv$, and we have
	\[
	R_H(u,v)\leq R_C(u,v)=\frac{t-1}{t},
	\]
	which contradicts $	R_H(u,v)=1$. Therefore, for any $u, v \in V(H)$, if $R(u, v) = 1$, then $uv \notin E(H)$.
	
	Since the multiplicity of $\frac{2}{n}$ in $RS(H)$ is $1$, let $R(u_0, v_0) = \frac{2}{n}$. If $u_0 v_0 \in E(H)$, then by Lemma \ref{lemma:n-1} we have:
	\[
	\frac{2}{n}+(\frac{1}{2}+\frac{1}{2n})(|E(H)|-1)=\sum_{uv \in E(H)} R(u, v) = n + 1.
	\]
	Solving this yields $|E(H)|=2n+1-\frac{4}{n+1}$, which contradicts the fact that $|E(H)|$ is an integer. Hence, we have $|E(H)|=2n$, and for any $u, v \in E(H)$ we have $R(u,v)=\frac{1}{2}+\frac{1}{2n}$.
	
	Applying Lemma \ref{lamma:n-2} to $u_0, v_0$, we obtain
	\[
	\frac{2}{n}d(u_0)+\sum_{z \in N(u_0)}\left( \frac{1}{2}+\frac{1}{2n}-R(v_0,z) \right)=R(u_0,v_0)d(u_0)+\sum_{z \in N(u_0)}\left( R(u_0,z)-R(v_0,z) \right)=2.
	\]
	
	Since for any $z \in N(u_0)$, the possible values of $R(v_0, z)$ are only $1$ and $\frac{1}{2}+\frac{1}{2n}$, we have
	\[
	2=\frac{2}{n}d(u_0)+\sum_{z \in N(u_0)}\left( \frac{1}{2}+\frac{1}{2n}-R(v_0,z) \right)\leq \frac{2}{n}d(u_0).
	\]
	Solving this gives $d(u_0)\geq n$. Due to $|V(H)|=n+2$ and $u_0,v_0\notin N(u_0)$, we have $d(u_0)=n$, i.e., $N(u_0)=V(H)\backslash \{u_0,v_0\}$ Similarly, $N(v_0)=V(H)\backslash \{u_0,v_0\}$. Thus, we have found a complete bipartite subgraph $H''$ of $H$ with bipartition $\{u_0,v_0\}$ and $N(u_0)$, and because $|E(H)|=|E(H'')|$, we have $H=H''\cong K_{2,n}$.
\end{proof}

When $m$ is much larger than $n$, the differences among $\frac{2}{m}$, $\frac{1}{m}+\frac{1}{n}-\frac{1}{mn}$, and $\frac{2}{n}$ become more pronounced. This allows us to take advantage of an important property of resistance distance---the triangle inequality property, namely Lemma \ref{lemma:Triangle inequality}. Finally, we prove that $K_{m,n}$ can be determined by its resistance spectrum when $m > 3n + 1$. This result demonstrates that most complete bipartite graphs can be determined by their resistance spectra.

\begin{theorem}
		When $m>3n+1$, the complete bipartite graph $K_{n,m}$ is determined by its resistance spectrum.
\end{theorem}
\begin{proof}
	The resistance spectrum of the complete bipartite graphs $K_{n,m}$ is given by
	\[
	\RS(K_{n,m})=\left\lbrace \left[\frac{2}{m} \right]^{\frac{n(n-1)}{2}}, \left[ \frac{1}{m}+\frac{1}{n}-\frac{1}{mn}\right]^{mn}, \left[\frac{2}{n} \right]^{\frac{m(m-1)}{2}} \right\rbrace.
	\]
	
	If $H$ has the same resistance spectrum as the complete bipartite graph $K_{n,m}$, we prove below that $H$ is isomorphic to $K_{n,m}$.

	Take any three distinct vertices $u, v, x$ satisfying $R(u,v)=\frac{2}{m}$. Since $m>3n+1$, by Lemma \ref{lemma:Triangle inequality}, we must have $R(u,x)=R(v,x)$ holds. If $uv\in E(H)$, then applying Lemma \ref{lamma:n-2} to $u$ and $v$ yields
	\[
	\frac{2}{m}(d(u)+1)=R(u,v)d(u)+R(u,v)+\sum_{z \in N(u)\backslash\{v\}}\left( R(u,z)-R(v,z) \right)=2.
	\]
	Solving this gives $d(u)=m-1$. Similarly, $d(v)=m-1$, which clearly satisfies
	\[
	R(u,v)=\frac{2}{m}=\frac{1}{d(u)+1}+\frac{1}{d(v)+1}.
	\]
	By the equality condition in Lemma \ref{lemma:xiajie}, we have $N(u)\backslash \{v\}=N(v)\backslash\{u\}$. If $uv\notin E(H)$, then applying  Lemma \ref{lamma:n-2} to $u$ and $v$ yields
	\[
	\frac{2}{m}d(u)=R(u,v)d(u)+\sum_{z \in N(u)}\left( R(u,z)-R(v,z) \right)=2.
	\]
	Solving this gives $d(u)=m$. Similarly, $d(v)=m$. Adding an edge $uv$ to $H$ results in a new graph $H'$. By the parallel principle and the principle of substitution, we have $R_{H'}(u,v)=\frac{2}{m+2}$, and thus
	\[
	R_{H'}(u,v)=\frac{2}{m+2}=\frac{1}{d_{H'}(u)+1}+\frac{1}{d_{H'}(v)+1}.
	\]
	By the equality condition in Lemma \ref{lemma:xiajie}, we obtain $N_{H'}(u)\backslash \{v\}=N_{H'}(v)\backslash\{u\}$, so $N_H(u)=N_H(v)$ holds.
	
	The above discussion proves that for any two vertices $u, v$ satisfying $R(u,v)=\frac{2}{m}$, either $d(u)=d(v)=m-1$ and $N(u)\backslash \{v\}=N(v)\backslash\{u\}$; or $d(u)=d(v)=m$ and $N(u)=N(v)$.
	
	Define the set $A = \{ u \in V(H) | \exists v \in V(H) \text{ such that } R(u,v) = \frac{2}{m} \}$. Denote $B = V(H) \setminus A$. Then, from the above discussion, for any $u \in A$, we have $d(u) = m$ or $m-1$. Based on the multiplicity of $\frac{2}{m}$ in $RS(H)$, we can determine $|A| \geq n$. We will show that for any $u \in A$, there exists a vertex $x \in N(u)$ such that $R(u,x)\neq \frac{2}{m}$. 
	
	If $d(u) = m$, i.e., there exists a vertex $v \in A$ such that $R(u,v)=\frac{2}{m}$. If there exists a vertex $x$ such that $R(u,x)=\frac{2}{m}$ and $ux\in E(H)$ hold, then $d(u) = m-1$, which contradicts $d(u) = m$. Therefore, for any $u \in A$, if $d(u) = m$, then there must exist a vertex $x \in N(u)$ such that $R(u,x)\neq \frac{2}{m}$.
	
	If $d(u) = m-1$, assume for contradiction that for every $x \in N(u)$, $R(u,x)= \frac{2}{m}$ holds. Then, applying Lemma \ref{lamma:n-2} to all $x \in N(u)$ and $u$, we obtain $d(u)=d(x)=m-1$ and $N(u)\backslash \{x\}=N(x)\backslash\{u\}$. That is, $H$ contains a complete subgraph composed of $U=N(u)\cup \{u\}$. Since $|U|=m$, the vertex set $V(H)\setminus U$ is not empty. Take any $x\in V(H)\setminus U$; clearly $x$ is not adjacent to $u$, and due to the equation, $x$ is not adjacent to any vertex in $U$. By the arbitrariness of $x$, it follows that $H$ is a disconnected graph, which leads to a contradiction.
	
	Hence, we have shown that for any $u \in A$, there exists a vertex $x \in N(u)$ such that $R(u,x)\neq \frac{2}{m}$. Next, we prove that for any $u\in A$ and any $v\in V(H)$, we must have $R(u,v)\neq \frac{2}{n}$.
	
	For any $u \in A$ and $v \in V(H)$, if $uv\in E(H)$ and $R(u,v)=\frac{2}{n}$, then by Lemma \ref{lamma:n-2}, we have
	\[
	\frac{2}{n}d(u)+\frac{2}{n}+\sum_{z \in N(u)\backslash\{v\}}\left( R(u,z)-R(v,z) \right)=2.
	\]
	Thus,
	\[
	\frac{2}{n}(d(u)+1)+(\frac{2}{m}-\frac{2}{n})(d(u)-1)\leq 2.
	\]
	Solving this gives $d(u)\leq m-\frac{2m-n}{n} \leq m-\frac{m}{n}<m-1$, which contradicts $d(u)\geq m-1$. Therefore, for any $u \in A$ and $v \in V(H)$, if $uv\in E(H)$, then $R(u,v)\neq \frac{2}{n}$. 
	
	For any $u \in A$ and $v \in V(H)$, if $uv\notin E(H)$ and $R(u,v)=\frac{2}{n}$, then by Lemma \ref{lamma:n-2}, we have
	\[
	\frac{2}{n}d(u)+\sum_{z \in N(u)}\left( R(u,z)-R(v,z) \right)=2.
	\]
	Since $u \in A$, there exists a vertex $z \in N(u)$ such that $R(z,u)\neq\frac{2}{m}$, i.e., $R(z,u)\geq \frac{1}{m}+\frac{1}{n}-\frac{1}{mn}$. Thus,
	\[
	\frac{2}{n}d(u)- (\frac{2}{n}-( \frac{1}{m}+\frac{1}{n}-\frac{1}{mn} ) )-(\frac{2}{n}-\frac{2}{m})(d(u)-1)\leq 2.
	\]
	Solving this yields $d(u)\leq m-\frac{m-n-1}{2n}<m-1$, which contradicts $d(u)\geq m-1$. Therefore, for any $u \in A$ and $v \in V(H)$, if $uv\notin E(H)$, then $R(u,v)\neq \frac{2}{n}$.
	
	The above discussion proves that for any $u\in A$ and  $v \in V(H)$, we must have $R(u,v)\neq \frac{2}{n}$. That is, if $R(u,v)= \frac{2}{n}$, then $u,v\in B$. Therefore, based on the multiplicity of $\frac{2}{n}$ in $RS(H)$, we can determine $|B| \geq m$. Since $|A|\geq n$ and $|A|+|B|=m+n$, we must have $|A|=n$ and $|B|=m$. This implies that for any $u, v \in A$, we have $R(u,v)=\frac{2}{m}$. Consequently, for any $u, v \in A$, we have $d(u)=d(v)$.
	
	If for every $u\in A$, $d(u)=m-1$ is true, then it is easy to deduce $N(u)=A\backslash\{u\}$. This contradicts the fact that for any $u \in A$, there exists a vertex $x \in N(u)$ such that $R(u,x)\neq \frac{2}{m}$. Therefore, for every $u \in A$, $d(u)=m$ holds. Thus, for any $u, v \in A$, we have $uv\notin E(H)$ and $N(u)=N(v)=B$. Combing with the following of the set $A$, we can conclude that for any $u\in A$ and $v\in B$, we have $R(u,v)=\frac{1}{m}+\frac{1}{n}-\frac{1}{mn}$. By Lemma \ref{lemma:n-1}, we have
	\[
	m+n-1=mn(\frac{1}{m}+\frac{1}{n}-\frac{1}{mn})\leq \sum_{uv \in E(H)} R(u, v) = m+n - 1.
	\]
	The equation holds if and only if for any $u, v \in B$, $uv\notin E(H)$ holds. Hence, we have proven that $H$ is a complete bipartite graph with bipartition $A$ and $B$, i.e., $H \cong K_{n,m}$.
\end{proof}

\section{Conclusion}\label{sec:Conclusion}
In this paper, we investigate the problem of whether complete bipartite graphs can be determined by their resistance spectra. We show that $K_{n,n}$, $K_{n,n+1}$, and $K_{m,n}$ (with $m > 3n+1$) are all determinable by their resistance spectra. These results naturally lead to the conjecture that \textit{all} complete bipartite graphs may be determinable by their resistance spectra. Proving this general case or providing a counterexample demonstrating that certain classes of complete bipartite graphs are not determinable by their resistance spectra would be an interesting problem for future research.
\begin{conjecture}
	Every complete bipartite graph \(K_{m,n}\) can be determined by its resistance spectrum.
\end{conjecture}



\bibliography{references1}

@article{KR1993,
	author = {Klein, D. J. and Randi\'{c}, M.},
	title = {Resistance distance},
	journal = {J. Math. Chem.},
	DOI = {10.1007/bf01164627},
	volume = {12},
	number = {1},
	pages = {81-95},
	year = {1993},
	type = {Journal Article}
}

@article{H2010,
	author = {Haiyan Chen},
	title = {Random walks and the effective resistance sum rules},
	journal = {Discrete Appl. Math.},
	DOI = {10.1016/j.dam.2010.05.020},
	volume = {158},
	number = {15},
	pages = {1691–1700},
	year = {2010},
	type = {Journal Article}
}

@incollection{foster1949average,
	author    = {Foster, Ronald M},
	title     = {The average impedance of an electrical network},
	booktitle = {Contributions to Applied Mechanics (Reissner Anniversary Volume)},
	year      = {1949},
	pages     = {333},
	publisher = {Edwards Bros.},
	address   = {Ann Arbor, Michigan}
}

@book{BM2008graph,
	author = {Bondy, J. A. and Murty, U. S. R.},
	title = {Graph theory},
	publisher = {Springer-Verlag},
	address = {New York},
	series = {Grad. Texts In Math. 244},
	year = {2008},
	type = {Book}
}

@article{J1871,
	author = {Strutt, John William},
	title = {On the theory of resonance},
	journal = {Phil. Trans. R. Soc.},
    URL = {http://www.jstor.org/stable/109027},
	volume = {161},
	pages = {77--118},
	year = {1871},
	type = {Journal Article}
}

@article{SARDAR2020123076,
	title = {On the resistance diameter of hypercubes},
	journal = {Physica A: Statistical Mechanics and its Applications},
	volume = {540},
	pages = {123076},
	year = {2020},
	issn = {0378-4371},
	doi = {10.1016/j.physa.2019.123076},
	author = {Muhammad Shoaib Sardar and Hongbo Hua and Xiang-Feng Pan and Hassan Raza}
}

@article{klein2002resistance,
	title={Resistance-distance sum rules},
	author={Klein, Douglas J},
	journal={Croat. Chem. Acta},
	volume={75},
	number={2},
	pages={633--649},
	year={2002},
	publisher={Hrvatsko kemijsko dru{\v{s}}tvo}
}

@inproceedings{chandra1989electrical,
	title={The electrical resistance of a graph captures its commute and cover times},
	author={Chandra, Ashok K and Raghavan, Prabhakar and Ruzzo, Walter L and Smolensky, Roman},
	booktitle={Proceedings of the twenty-first annual ACM symposium on Theory of computing},
	DOI = {10.1145/73007.73062},
	pages={574--586},
	year={1989},
	publisherx = {Association for Computing Machinery},
	addressx = {New York, NY, USA}, 
	numpagesx = {13},
	locationx = {Seattle, Washington, USA}, 
	seriesx = {STOC '89} 
}

@inproceedings{qiu2006robust,
	title={Robust multi-body motion tracking using commute time clustering},
	author={Qiu, Huaijun and Hancock, Edwin R},
	booktitle={Computer Vision--ECCV 2006: 9th European Conference on Computer Vision},
	DOI = {10.1007/11744023_13},
	pages={160--173},
	year={2006},
	organizationx={Springer},
	editorx={Leonardis, Ale{\v{s}} and Bischof, Horst and Pinz, Axel},
	publisherx={Springer Berlin Heidelberg},
	addressx={Berlin, Heidelberg},
	isbnx={978-3-540-33833-8}
}

@inproceedings{liben2003link,
	title={The link prediction problem for social networks},
	author={Liben-Nowell, David and Kleinberg, Jon},
	booktitle={Proceedings of the twelfth international conference on Information and knowledge management},
	DOI = {10.1145/956863.956972},
	pages={556--559},
	year={2003},
	seriesx = {CIKM'03},
	publisherx = {Association for Computing Machinery},
	addressx = {New York, NY, USA},
	numpagesx = {4},
	locationx = {New Orleans, LA, USA}
}

@inproceedings{qiu2005image,
	title={Image Segmentation using Commute Times},
	author={Qiu, Huaijun and Hancock, Edwin R},
	booktitle = {Proceedings of the British Machine Vision Conference},
	DOI = {10.5244/C.19.94},
	pages={94.1-94.10},
	year={2005},
	publisherx = {BMVA Press}
}

@article{zhu2024resistanceDirected,
	author = {Zhu, Mingzhe and Zhu, Liwang and Li, Huan and Li, Wei and Zhang, Zhongzhi},
	title = {Resistance Distances in Directed Graphs: Definitions, Properties, and Applications},
	journal = {Theoret. Comput. Sci.},
	doi = {10.1016/j.tcs.2024.114700},
	volume = {1009},
	pages = {114700},
	year = {2024},
	type = {Journal Article}
}

@article{sardar2024clawfree,
	title={Computation of the resistance distance and the {K}irchhoff index for the two types of claw-free cubic graphs},
	author={Sardar, Muhammad Shoaib and Pan, Xiang-Feng and Xu, Shou-Jun},
	journal={Appl. Math. Comput.},
	doi = {10.1016/j.amc.2024.128670},
	volume={473},
	pages={128670},
	year={2024},
	publisher={Elsevier}
}

@misc{Baxter1999RS22, 
	author = {Baxter, L.},
	titlex = {Counterexamples Wanted--Graph Isomorphism \& Resistances},
	howpublishedx = {sci.math.research},
	yearx = {Apr. 22, 1999},
	note = {{C}ounterexamples wanted--graph isomorphism \& resistances. 
	\url{https://groups.google.com/group/
	sci.math.research/msg/5cee69ecf9edb863?dmode=source}, 1999 (accessed 7 July 2025)
	}
}

@misc{Rickard1999Counterexample23,
	author = {Rickard, J.},
	titlex = {Counterexample Wanted for Graph Isomorphism Conjecture},
	howpublishedx = {comp.theory},
	yearx = {Apr. 23, 1999},
	note = {{C}ounterexample wanted for graph isomorphism conjecture. 
	\url{https://groups.google.com/group/comp.theory/msg
	/53976638ed387d76?dmode=source}, 1999 (accessed 1 July 2025)
	}
}

@misc{WeissteinRSEquivalent,
	author = {Weisstein, E. W.},
	titlex = {Resistance-Equivalent Graphs. From MathWorld--A Wolfram Web Resource},
	howpublishedx = {From MathWorld--A Wolfram Web Resource},
	note = {{R}esistance-equivalent graphs. From MathWorld—a wolfram web resource. 
	\url{https://mathworld.wolfram.com/
	Resistance-EquivalentGraphs.html}, (accessed 1 July 2025)
	}
}

@article{XU2025114284,
	title = {A method for constructing graphs with the same resistance spectrum},
	journal = {Discrete Math.},
	volume = {348},
	number = {2},
	pages = {114284},
	year = {2025},
	issn = {0012-365X},
	author = {Si-Ao Xu and Huan Zhou and Xiang-Feng Pan},
	doi = {10.1016/j.disc.2024.114284}
}

@mastersthesis{Xu2021Resistance,
	author  = {Xu, Si-Ao},
	title   = {Research on several kinds of graph parameters based on resistance distance},
	school  = {Anhui University},
	year    = {2021},
	type    = {Master's thesis},
}

@mastersthesis{Xu2022Resistance,
	author  = {Xu, Hui},
	title   = {The Research on Resistance Distance in Some Graphs},
	school  = {Anhui University},
	year    = {2022},
	type    = {Master's thesis},
}

@article{xing2025determination,
	author = {Xing, Baohua and Sun, Minhao and Zhou, Huan and Pan, Xiang-Feng
	},
title = {Determination of some graphs by resistance spectra},
journal = {Comput. Appl. Math.},
DOI = {10.1007/s40314-025-03191-1},
volume = {44},
number = {},
pages = {237},
year = {2025},
type = {Journal Article}
}

@article{zhou2025determination,
	title     = {Determination of some trees and unicyclic graphs by resistance spectra},
	author    = {Zhou, Huan and Ni, Qi and Lu, Ning-Ning and Pan, Xiang-Feng},
	journal   = {Journal of Jimei University (Natural Science)},
	year      = {2025},
	volume    = {30},
	number    = {3},
	pages     = {286--291},
	publisher = {Jimei University} 
}

@article{10.1088/1402-4896/adf7df,
	author={Diao, Zhuo and Pan, Xiang-Feng and Liu, Xiang-Yang},
	title={Analysis of Some Graphs through Resistance Spectra∗},
	journal={Physica Scripta},
	url={http://iopscience.iop.org/article/10.1088/1402-4896/adf7df},
	year={2025}
}

@article{OTHMAN2024113865,
	title = {Almost fair perfect matchings in complete bipartite graphs},
	journal = {Discrete Mathematics},
	volume = {347},
	number = {4},
	pages = {113865},
	year = {2024},
	issn = {0012-365X},
	doi = {https://doi.org/10.1016/j.disc.2023.113865},
	url = {https://www.sciencedirect.com/science/article/pii/S0012365X23005514},
	author = {Abeer Othman and Eli Berger}
}

@inproceedings{inproceedings,
	author = {Zhao, Lingqi and Wang, Mujiangshan and Zhang, Xuefei and Lin, Yuqing and Wang, Shiying},
	year = {2017},
	month = {01},
	pages = {},
	title = {An Algorithm for the Orientation of Complete Bipartite Graphs},
	doi = {10.2991/ammsa-17.2017.81}
}

@article{Bekos2023OnTD,
	title={On the Deque and Rique Numbers of Complete and Complete Bipartite Graphs},
	author={Michael A. Bekos and Michael Kaufmann and Maria Eleni Pavlidi and Xenia Rieger},
	journal={ArXiv},
	year={2023},
	volume={abs/2306.15395},
	url={https://api.semanticscholar.org/CorpusID:259262070}
}

\end{document}